\documentclass[reqno]{amsart}

\usepackage{amsmath,amsxtra,amssymb,eufrak}
\usepackage[all]{xy}
\usepackage[active]{srcltx} 
\usepackage{mathrsfs}

\newtheorem{thm}{Theorem}[section]
\newtheorem{lm}[thm]{Lemma}
\newtheorem{co}[thm]{Corollary}
\newtheorem{pr}[thm]{Proposition}

\theoremstyle{definition}

\newtheorem{df}[thm]{Definition}

\newtheorem{exm}[thm]{Example}

\newtheorem{rem}[thm]{Remark}

\numberwithin{equation}{section}

\DeclareMathOperator{\Ker}{Ker}

\DeclareMathOperator{\LinC}{Lin_\CC}

\newcommand*{\Ptens}{\mathop{\widehat\otimes}}

\newcommand*{\wh}{\widehat}
\newcommand*{\wt}{\widetilde}

\newcommand*{\spn}{\mathrm{span}}

\newcommand{\Rad}{\mathop{\mathrm{Rad}}\nolimits}
\newcommand{\rad}{\mathop{\mathrm{rad}}\nolimits}

\newcommand{\ad}{\mathop{\mathrm{ad}}\nolimits}
\newcommand{\GL}{\mathop{\mathrm{GL}}\nolimits}
\newcommand{\SL}{\mathop{\mathrm{SL}}\nolimits}

\newcommand{\acv}{\mathop{\mathrm{a.c.}}\nolimits}

{\end{compactenum}}
\newcommand*{\cA}{\mathscr A}

\newcommand{\CC}{\mathbb{C}}

\newcommand{\R}{\mathbb{R}}
\newcommand{\Z}{\mathbb{Z}}
\newcommand{\N}{\mathbb{N}}

\newcommand{\cO}{\mathcal{O}}

\newcommand*{\cR}{\mathcal R}

\newcommand*{\fb}{\mathfrak{b}}

\newcommand*{\fj}{\mathfrak{j}}

\newcommand*{\fe}{\mathfrak{e}}
\newcommand*{\fg}{\mathfrak{g}}
\newcommand*{\fl}{\mathfrak{l}}

\newcommand*{\fh}{\mathfrak{h}}
\newcommand*{\fr}{\mathfrak{r}}
\newcommand*{\fs}{\mathfrak{s}}
\newcommand*{\ft}{\mathfrak{t}}
\newcommand*{\fv}{\mathfrak{v}}

\renewcommand{\le}{\leqslant}
\renewcommand{\ge}{\geqslant}

\let \de         =\delta
      
\let \ze         =\zeta
\let \te         =\theta
\let \io         =\iota

\let \la         =\lambda
\let \si         =\sigma
\let \up         =\upsilon
\let \om         =\omega

\let \phi         =\varphi

\title[Holomorphic functions of exponential type]
{Holomorphic functions of exponential type on connected complex
Lie groups} 
\author{O. Yu. Aristov}
\email{aristovoyu@inbox.ru}

\begin{document}
 \maketitle
\begin{abstract}
Holomorphic functions of exponential type on a complex Lie group
$G$ (introduced by Akbarov) form a locally convex algebra, which
is denoted by $\cO_{exp}(G)$. Our aim is to describe the structure
of $\cO_{exp}(G)$ in the case when $G$ is connected. The following
topics are auxiliary for the claimed purpose but of independent
interest: (1) a characterization of linear complex Lie group
(a~result  similar to that of Luminet and Valette for real Lie
groups); (2) properties of the exponential radical when $G$ is
linear; (3) an asymptotic decomposition of a word length function
into a sum of three summands (again for linear groups). The main
result presents $\cO_{exp}(G)$ as a complete projective tensor of
three factors, corresponding to the length function decomposition.
As an application, it is shown that if $G$ is linear then the
Arens-Michael envelope of $\cO_{exp}(G)$ is the algebra of all
holomorphic functions.
\end{abstract}

 \section{Introduction}

A holomorphic function on a complex Lie group $G$ is said to be of
exponential type if it is majorized by a submultiplicative weight
(a non-negative locally bounded function $\om$ satisfying
$\om(gh)\le \om(g)\om(h)$ for all $g, h \in G$). Akbarov in
\cite{Ak08} introduced this notion for a compactly generated Stein
group.  In fact, the definition can be used as well for a general
complex Lie group~$G$. The set $\cO_{exp}(G)$ of all holomorphic
functions of exponential type is a $\Ptens$-algebra (i.e., a
complete Hausdorff locally convex topological algebra with jointly
continuous multiplication) w.r.t. the point-wise multiplication
(see Lemma~\ref{Oexpptnal}). Considering basic examples, Akbarov
showed that $\cO_{exp}(\CC^m)$ coincides with the classical space
of entire functions  of exponential type, i.e., having  at most
order~$1$ and finite type, on $\CC^m$.  (The terminology is taken
from this example.) On the other hand, he proved that
$\cO_{exp}(\GL_m(\CC))$ is $\cR(\GL_m(\CC))$, the algebra of
regular functions in the sense of algebraic geometry. The main
objective of this text is to give a explicit description of
$\cO_{exp}(G)$ for an arbitrary connected complex Lie group~$G$.

Our interest is motivated by investigation of holomorphic
reflexivity for some topological Hopf algebras initiated in
[ibid.]. The essential question in this direction whether or not
the natural map from $\cO_{exp}(G)$ to $\cO(G)$, the algebra of
all holomorphic functions on~$G$, is an Arens-Michael envelope, in
other words, whether or not $\cO(G)$ is topologically isomorphic
to the completion of $\cO_{exp}(G)$ w.r.t. the topology determined
by all possible continuous submultiplicative prenorms. It is
claimed in [ibid., Lem.~6.6] that this is true if $G$ is affine
algebraic and connected. Unfortunately, the argument contains a
gap. (In the proof, two maps, $\rho$ and $\wt\rho$, are considered
but it is not clear why $\wt\rho$ extends $\rho$.) This defect can be
fixed if one shows that $\cR(G)\to \cO_{exp}(G)$ has dense range.
This is possible but the only way to prove density that I know is
by rude force, i.e., applying the main structural result of this
article (Theorem~\ref{fexpdec}); cf. also Theorem~\ref{expclin}.

To distinguish functions of exponential type inside $\cO(G)$ we
need to understand asymptotic behavior of a word length function.
The reason is that any submultiplicative weight has the form
$g\mapsto e^{\ell(g)}$, where $\ell$ is a length function, and any
length function is dominated at infinity  by a word length
function. We will be concerned with growth rate of any (eq.,
every) word length function in Section~\ref{ABWLF}.

Two normal closed subgroups of $G$, the linearizer and the
exponential radical, appear  naturally in the exposition. To
appreciate why they are important consider the following examples.

If $H$ is the $3$-dimensional complex Heisenberg group, which can
be presented as
\begin{equation}\label{3Hei}
\begin{pmatrix}
 1& a& b\\
 0&1 & c\\
 0 & 0&1
\end{pmatrix}\qquad (a,b,c\in\CC)\,,
\end{equation}
then, as not hard to see,   a word length function is equivalent
to $|a|+|c|+|b|^{1/2}$; in particular, all polynomials in $a$,
$b$, and $c$ are of exponential type. Similar asymptotic behavior
for an arbitrary simply connected nilpotent group $G$ is found in
\cite{VSC92} and \cite{Kar94}. In \cite{ArAMN}, which can be
considered as the first part of this text, these results are used
to determine the structure of $\cO_{exp}(G)$ for a general simply
connected nilpotent complex Lie group $G$.

On the other hand, consider the quotient of the Heisenberg group
over the discrete central subgroup given by
$$
N\!:=\begin{pmatrix}
1& 0& n\\
 0&1 & 0\\
 0 & 0&1
\end{pmatrix}\qquad (n\in\Z)\,.
$$
It is identified  with $\CC^\times\times\CC^2$ endowing with the
group law
$$
 (z,a,c)\cdot (z',a',c')\!:= (zz'e^{ca'},a+a',   c+c')\,.
$$
It is easy to see that the coordinate functions $a$ and $c$ are of
exponential type. Nevertheless $z$ is not of exponential type.
Indeed, take the  length function $\ell$ associated with the
symmetric neighbourhood of the identity $U\!:=\{1/2\le|z|\le 2,
\,|a|,\,|c|\le 1\}$. If $z$ is dominated by $e^\ell$, then there
are constants $C$ and $D\ge  0$ s.t. $|z(g)| \le e^{C\ell(g) + D}$
for all $g\in H/N$. Set $h=(1,0,1)$ and $g=(1,1,0)$. Then
$h^ng^n=(e^{n^2},n,n)$; hence $z(h^ng^n)= e^{n^2}$. Since $h$ and
$g$ are in $U$,  we get $\ell(h^ng^n)\le 2n$. Hence $e^{n^2}\le
e^{2Cn + D}$ for all $n\in\N$,  a contradiction.

Moreover, it can be shown that $\cO_{\exp}(H/N)$ is isomorphic to
$\cO_{\exp}(\CC^2)$, so the dimension degenerates. Crucial observation
to generalize this argument is that $H/N$ is not linear as a
complex Lie group. (Actually $H/N$ is a standard example of a
non-linear complex Lie group.) In fact, if $G$ is connected, then
we have an isomorphism $\cO_{exp}(G/\LinC(G))\cong\cO_{exp}(G)$,
where $\LinC(G)$ is the linearizer of $G$ (the intersection of
kernels of all finite-dimensional holomorphic representations);
see Theorem~\ref{redliexpf}. For the proof, we need an auxiliary
result: \emph{If holomorphic homomorphisms of a connected complex
Lie group $G$ to invertibles of Banach algebras separate points,
then $G$ is linear} (see Theorem~\ref{holhomline}).

Further, consider another example: the  simply connected
$2$-dimensional non-abelian Lie group $S$, i.e.,  $\CC^2$ with the
multiplication
\begin{equation}\label{2dimmul}
(s, t)\cdot (s', t')\!:=(s + s', te^{s'} + t')\,.
\end{equation}
It is not hard to see  that any word length function on $S$ is
equivalent to $|s|+\log|t|$, where $(s,t)\in S$. This
decomposition corresponds to the presentation
$$
\cO_{exp}(S)\cong
\cO_{exp}(\CC)\otimes \cR(\CC)\,.
$$

To transfer this observation to the general case we need notion of
exponential radical. The idea of exponential radical dates back to
Guivarc'h \cite{Gu80}. It was rediscovered and named by Osin in
\cite{Os02}, where the simply connected solvable case is
considered. In \cite{Co08},  Cornulier modified Osin's definition
in a way more convenient to general connected Lie groups.  The
main property of the exponential radical is that it is a strictly
exponentially distorted subgroup, which means logarithmic growth
for the restriction of a word length function (e.g., the
exponential radical of the group $S$ defined in~\eqref{2dimmul} is
$\{(0,t)\!:\,t\in\CC\}$).

For a complex Lie group, the exponential radical is easier to
describe than in the real case and we consider it carefully in
Section~\ref{ERLCG}.

\subsection*{Acknowledgements}
The author is grateful to  K.-H.~ Neeb for valuable
comments.

 \section{Characterization of the linearizer}\label{CL}

Our main references to the structure theory of Lie groups are
\cite{Ho65,Le02,HiNe}.  Our terminology and notation in this area
are principally from \cite{HiNe}. For a complex Lie group~$G$, the
intersection of kernels of all finite-dimensional holomorphic
representations is called the \emph{linearizer} and is denoted by
$\LinC(G)$.  Also, $G$~is  \emph{linear} if it admits a faithful
finite-dimensional holomorphic representation (eq.,
$\LinC(G)=\{1\}$). Further, $G$~is  \emph{linearly complex
reductive} if there exists a compact real Lie group~$K$ s.t.~$G$
is the universal complexification of~$K$ \cite[Def.~15.2.7]{HiNe}
(it is called 'reductive' in \cite{Le02}; cf. [ibid. Th. 4.31]).
An \emph{integral subgroup}~$H$ of~$G$ is a subgroup that is
generated by $\exp \fh$ for a subalgebra~$\fh$ of the Lie algebra
of~$G$; in this case we write $H=\langle\exp \fh\rangle$. If $H$
is a  closed subgroup and $G$ is connected, then $H^*$ denotes the
smallest complex integral subgroup containing~$H$ \cite[Defs.~
9.4.10, 15.2.11]{HiNe}.  We write $(G,G)$ for the subgroup
generated by the commutators $ghg^{-1}h^{-1}$ for $g,h\in G$. The
connected component of~$1$ and the center of~$G$ are denoted by
$G_0$ and $Z(G)$, resp.

It is known that a connected real Lie group with Levi-Malcev
decomposition $G=RS$, where $R$ is the solvable radical and $S$ is
a semisimple Levi factor, is linear iff  $R$ and $S$  are both
linear (as real Lie groups) \cite[Th. XVIII.4.2]{Ho65}. Since a
connected semisimple complex Lie group is always linear \cite[Th.
XVII.3.2]{Ho65}, the following proposition is an analogue of this
theorem in the complex case. The author was unable to find a
proof of this result in the literature.

\begin{thm}\label{linrad}
A connected complex Lie group is linear iff its radical is linear.
\end{thm}
\begin{proof}
The necessity is evident. To prove the sufficiency take a
connected complex Lie group $G$ and consider a Levi-Malcev
decomposition $G=RS$, where $R$ is the solvable radical and $S$ is
a semisimple Levi subgroup. Suppose that $R$ is linear. Then
$R=B\rtimes L$, where $B$  is simply connected solvable and $L$ is
linearly complex reductive \cite[Th.~16.3.7]{HiNe}.

Our first goal is to replace $B$ by a normal subgroup of $G$ with
the same properties. It is not hard to show that we can assume
that $(S,L)=\{1\}$ and there is a normal integral subgroup  $B_1$
in $G$ s.t. $R=B_1L$, $(R,R)\subset B_1$, and $B_1\cap L$ is
discrete.  (The argument is almost the same as for the real case
in \cite[Th.~XVIII.4.2, paragraphs 1--3 of the proof]{Ho65}. The
only step which is different is that we have to apply the fact
that a holomorphic finite-dimensional representation of a
complexified torus is completely reducible
\cite[Th.~15.2.10]{HiNe} instead of that a finite-dimensional
representation of a real torus is completely reducible.)

We claim that $R=B_1\rtimes L$ and $B_1$ is simply connected
closed (cf.~ the second part of the proof for
\cite[Th.~XVIII.4.2]{Ho65}). Indeed, $R$ is a linear complex Lie
group and $L=K^*$ for some  maximal compact subgroup $K$  in $R$.
Hence $L\cap (R,R)=\{1\}$ \cite[Th.~16.3.7]{HiNe} and $(R,R)$ is
closed in $R$ \cite[Pr.~4.37]{Le02}. (Remark that the latter is
true for all linear real Lie groups \cite[Cor.~16.2.8]{HiNe}.)
Also, $(R,R)$ is normal in~$R$, so is $L(R,R)$. Since $R$ is
solvable, $L$ is abelian and contains a maximal torus of $R$.
Further, $L(R,R)$ is integral and contains a maximal torus, so, by
\cite[Cor.~14.5.6]{HiNe}, we get that $L(R,R)$ is closed in $R$.
Remind that $(R,R)\subset B_1$, so we can consider a homomorphism
of Lie groups $\si\!:B_1/(R,R)\to R/(L(R,R))$. Since $R=B_1L$, we
obtain that $\si$ is surjective.

On the other hand, $R=B\rtimes L$ and $L$ is abelian; hence$(R,R)$
is contained in $B$  and  $R/(L(R,R))$ is isomorphic to $B/(R,R)$.
Moreover, $(R,R)$ is a connected subgroup in the simply connected
group $B$; therefore  $(R,R)$ and $R/(L(R,R))$ are simply
connected. Further, $$\Ker \si=B_1/(R,R)\cap (L(R,R))/(R,R)$$ is
discrete because $B_1\cap L$ is discrete. Whence $\Ker \si$ is
trivial since $R/(L(R,R))$  is simply connected. Thus $\si$ is an
isomorphism. Hence  $B_1/(R,R)$ is simply connected; so is $B_1$.
Besides, it follows from
$$B_1/(R,R)\cap(L(R,R))/(R,R)=\{1\}$$ that $B_1\cap
L(R,R)\subset(R,R)$. Consequently $B_1\cap L\subset
L\cap(R,R)=\{1\}$; in particular, $B_1$ is closed.  The claim is
proved.

Since $(S,L)=\{1\}$, the set $LS$ is an integral subgroup in $G$.
By \cite[Th.~16.3.7]{HiNe}, to complete the proof we need to show
that $LS$ is linearly complex reductive and $G=B_1\rtimes LS$.
First, we claim that $G$ is a Stein group. Indeed, since $R$ is
linear, it is a Stein group. On the other hand,  $Z(G)_0$ is a
closed subgroup of $R$, so $Z(G)_0$ is a Stein group. It follows
from the Matsushima-Morimoto theorem \cite[Th.~XIII.5.9]{Nee} that
a connected complex Lie group  is a Stein group iff the connected
component of $1$ in the center is a Stein group. The claim is
proved.

Further, $LS$ is an integral subgroup in a Stein group, therefore
it is  holomorphically separable, hence, by the
Matsushima-Morimoto theorem, it is a Stein group.  Since $L$ is
abelian, we get $Z(LS)=LZ(S)$. The center of a connected
semisimple complex Lie group is finite \cite[Cor.~4.17]{Le02}, so
$Z(LS)_0/L$ is finite. Then $Z(LS)_0$ is toroidal, i.e.,
$Z(LS)_0=T^*$ for some maximal torus~$T$ \cite[Pr.~15.3.9]{HiNe}.
Besides,  $Z(LS)_0$ is a Stein group because it is a closed
subgroup of a Stein group. It follows from \cite[Pr.~15.3.4]{HiNe}
that a toroidal Stein group is a complexified  torus. Since $LS$
is connected,  application of \cite[Th.~15.2.9]{HiNe} yields that
$LS$ is linearly complex reductive.

By dimension argument, $B_1\cap LS$ is discrete; hence it is
central. Let $g\in B_1\cap LS$. Then $g\in LZ(S)$ and $Z(S)$ is
finite, whence there is $k\in\N$ s.t. $g^k\in L$. Therefore
$g^k\in B_1\cap L=\{1\}$. The center of $B_1$ is simply connected,
so the only element of finite order is $1$. Thus $B_1\cap
LS=\{1\}$ and we have finally that $G=B_1\rtimes LS$.
\end{proof}

It is a standard fact that a Banach space valued function that is
 weakly holomorphic is also holomorphic w.r.t. the
norm; see, e.g. \cite[Th.~2.1.3]{He89}.  So we can say freely that
a homomorphism $\pi\!:G\to \GL(A)$, where $G$ is a complex Lie
group and $\GL(A)$ is the group of invertible elements of a unital
Banach algebra $A$, is \emph{holomorphic} if for any continuous
linear functional $x$ on $A$ the function $G\to\CC\!:g\mapsto
\langle x, \pi(g)\rangle$ is holomorphic.

\begin{thm}\label{holhomline}
The linearizer $\LinC(G)$ of a connected complex Lie group $G$
coincides  with
$$
\bigcap_\pi \{\Ker \pi\!: \,\pi\!: G\to \GL(A)\}\,,
$$
where $A$ runs all unital Banach algebras and $\pi$ runs all
possible holomorphic homomorphisms.
\end{thm}
\begin{rem}
Luminet and Vallete proved a result that gives a characterization
of the real linearizer for a real Lie group as the mutual kernel
of all norm continuous homomorphisms to  the invertible groups of
unital Banach algebras \cite[Th.~A,
(i)$\Leftrightarrow$(v)]{LuVa}. (Also, this is true for more
general class of continuous inverse algebras \cite{BN08}.) Since
any holomorphic homomorphism from a complex Lie group is
holomorphic w.r.t. the norm,  it is norm continuous. Thus we can
consider Theorem~\ref{holhomline}, in which the norm continuity
assumption is redundant,  as an analogue of the result of Luminet
and Vallete.
\end{rem}

For the proof, we need two lemmas and the notation $\Rad A$ for
the Jacobson radical of a Banach algebra~$A$. Recall that $a\in A$
is called \emph{topologically nilpotent} if $\|a^n\|^{1/n}\to 0$.

\begin{lm}\label{sol1pr}
Let $G$ be a connected solvable complex Lie group, $A$ a unital
Banach algebra, and $\pi\!:G\to \GL(A)$  a holomorphic
homomorphism. Then for each $g\in (G,G)$ there is a topologically
nilpotent $r\in A$
s.t. $\pi(g)=1+r$.
\end{lm}
\begin{proof}
Consider $\GL(A)$ as a (complex) Banach Lie group \cite[Exm.
III.1.11(b)]{Ne05} and $\pi$ as a Banach Lie group homomorphism.
Identifying the Lie algebra of $\GL(A)$ with $A$  we can write the
exponential map as
$$
\exp\!:A\to \GL(A)\!:a\mapsto \sum_{n=0}^\infty \frac{a^n}{n!}
$$
[ibid., Rem. IV.2.2]. Denote by $\fg$ the Lie algebra of $G$. Then
applying the Lie functor to $\pi$ we get a Banach Lie algebra
homomorphism $\mathrm{L}_\pi\!:\fg\to A$ s.t. $\pi\exp=\exp
\mathrm{L}_\pi$.

It follows from results of \cite{Tu84} (see also
\cite[Th.~24.1]{BS01}) that $\mathrm{L}_\pi[\fg,\fg]\subset\Rad
A_0$, where $A_0$ is the closed associative unital subalgebra of
$A$ generated by the solvable Lie subalgebra $\mathrm{L}_\pi \fg$
of $A$.  For any $\xi$ in  $[\fg,\fg]$  we have
$$
\exp \mathrm{L}_\pi(\xi)=\sum_{n=0}^\infty
\frac{\mathrm{L}_\pi(\xi)^n}{n!}\,.
$$
So $\exp \mathrm{L}_\pi(\xi)=1+r$ for some $r\in\Rad A_0$ (because
$\Rad A_0$ is closed). Since $G$ is connected, the subgroup
$(G,G)$ is generated by $\exp [\fg,\fg]$ \cite[Pr.~11.2.3]{HiNe}.
Therefore $\pi(g)$ has the same form for any $g\in (G,G)$.
Finally, note that each element  of the Jacobson radical of a
Banach algebra is topologically nilpotent \cite[Th.~2.1.33]{X2}.
\end{proof}

Recall that  a Banach algebra is said to be \emph{classically
semisimple} if it isomorphic to a finite product of full matrix
algebras over $\CC$. We denote by $\mathscr{E}(K)'$ the algebra of
distribution on a compact Lie group $K$.

\begin{lm}\label{cLgrfg}
Let $K$ be a compact  Lie group and let $A$ be a unital Banach
algebra. If $\phi\!:\mathscr{E}(K)'\to A$ is a continuous
homomorphism with dense range, then~$A$ is classically
semisimple.
\end{lm}
\begin{proof}
The linear space  $T$ of matrix coefficients of finite-dimensional
representations of $K$ is a (non-unital) subalgebra of
$\mathscr{E}(K)'$.  It follows  from  the Peter-Weyl Theorem that
$T$ is dense in $L^2(K)$.  But $L^2(K)$  is dense in
$\mathscr{E}(K)'$; therefore, $\phi(T)$ is dense in $A$. So we can
take $t\in T$ s.t. $\phi(t)$ is sufficiently close to $1$ to be
invertible.

It is well known (e.g.   \cite[Th.~27.21]{HeRo} ) that each element
of $T$ is contained in a finite sum of complemented  minimal
two-sided ideals and each such ideal is a full matrix algebra.
Therefore there exists a central idempotent $p\in T$ s.t. $tp=t$.
Then $1=\phi(t)^{-1}\phi(t)=\phi(t)^{-1}\phi(t)\phi(p)=\phi(p)$.
So $\phi(Tp)$  is dense in~$A$. Since $Tp$ is finite-dimensional,
this image is closed; hence $\phi(Tp)=A$ . Thus $A$~is a quotient
of the classically semisimple algebra $Tp$, so~$A$ is classically
semisimple itself.
\end{proof}

\begin{proof}[Proof of Theorem~\ref{holhomline}]
Since $\bigcap_\pi \Ker \pi\subset \LinC(G)$, it suffices to
consider the case where  $\bigcap_\pi \Ker \pi=\{1\}$ and prove
that $G$ is linear. By Theorem~\ref{linrad}, we can assume also
that $G$ is solvable. If follows from
\cite[Th.~16.3.7(v)$\Rightarrow$(iii)]{HiNe} that we need to check
that $K^*$ is linear and $K^*\cap (G,G)=\{1\}$ for any maximal
compact subgroup~$K$ of~$G$ (where $K^*$ is the smallest complex
integral subgroup containing~$K$ \cite[Def.~15.2.11]{HiNe}).

First, we claim that $G$ is a Stein group. Indeed, $\bigcap_\pi
\Ker \pi=\{1\}$ implies that coefficients of holomorphic
homomorphisms to Banach algebras, which are holomorphic functions,
separate points of $G$. Thus $G$ is holomorphically separable;
hence it is a Stein group \cite[Th.~XIII.5.9]{Nee}.

From \cite[Lem.~14.3.3]{HiNe} it follows that a maximal compact
subgroup $K$ of a solvable Lie group is abelian. Then $K^*$ is an
abelian integral subgroup. Whence $K^*$ contains a maximal torus;
so it is closed \cite[Cor.~4.5.6]{HiNe}.  Being a closed
submanifold of a Stein manifold, $K^*$ is a Stein manifold
\cite[Th.~XIII.5.2]{Nee}. Therefore, $K^*$ has no compact factors.
Hence, by \cite[Pr.~15.3.4(i)]{HiNe},  $K^*$ is a universal
complexification of $K$.

Now, let $A$ be a unital Banach algebra and $\pi\!:G\to \GL(A)$ a
holomorphic homomorphisms.  The restriction of $\pi$ to $K$ is
infinitely differentiable, hence it can be extended to a
continuous homomorphism $\phi\!:\mathscr{E}(K)'\to A$. Denote by
$C$ the closure of  $\phi(\mathscr{E}(K)')$ in $A$; then
$\pi(K)\subset \GL(C)$.  Lemma~\ref{cLgrfg} implies that $C$ is
classically semisimple Banach algebra. It is evident that $C$ is
commutative; therefore it is a finite sum of copies of $\CC$. So
$\GL(C)$ is a finite-dimensional Lie group. By the definition of
universal complexification, the homomorphism $K\to\GL(C)$ is
extended to a holomorphic homomorphism $K^*\to\GL(C)$, which
coincides with the restriction of $\pi$ by the uniqueness
property.

Let $g\in K^*\cap (G,G)$. By Lemma~\ref{sol1pr}, there exists a
topologically nilpotent $r\in A$ such that $\pi(g)=1+r$. Since
$\pi(g)\in C$, we have $r\in C$. But $C$ is semisimple and
commutative; so the only topologically nilpotent element in $C$
is~$0$ \cite[Pr.~2.1.34]{X2}. Therefore $\pi(g)=1$. Since $\pi$ is
arbitrary, it follows from the assumption that $g=1$.
\end{proof}

 \section{Exponential radical of a linear complex group}\label{ERLCG}

Recall that a (real) Lie group $G$ is of polynomial growth iff its
Lie group $\fg$ is of \emph{Type~R} \cite[6.25, 6.39]{Pat}, i.e.,
the eigenvalues of $\ad \xi$ are contained in $i\R$ for all
$\xi\in\fg$. The following theorem describes the exponential
radical of simply connected solvable Lie groups; it is a
combination of results from \cite{Gu80} and \cite{Os02}.  (See the
definition of a strictly exponentially distorted subgroup in
\eqref{sedustde} below.)
\begin{thm}\label{GuOs}
Let $G$ be a simply connected solvable Lie group and let $\fg$~be
the Lie  algebra of~$G$.

\emph{(A)} Then there exist a closed normal subgroup~$E$ s.t.
$G/E$ is the largest  polynomial growth quotient of~$G$  and a Lie
ideal~$\fe$ in~$\fg$ of~$G$  s.t. $\fg/\fe$ is the largest Type~R
quotient of~$\fg$. Moreover, $E=\langle \exp \fe\rangle$.

\emph{(B)} The subgroup~$E$ is nilpotent and stable under
automorphisms of~$G$; the ideal~$\fe$ is nilpotent and stable
under automorphisms of~$\fg$.

\emph{(C)} The subgroup $E$ is strictly exponentially distorted
in~$G$.
\end{thm}

Following \cite[Def.~6.2]{Co08}, we say that the \emph{exponential
radical} of a connected (real or complex) Lie group $G$ is  the
closed normal subgroup $E$ s.t. $G/E$ is the largest quotient of
$G$ that is a \emph{P-decomposed} Lie group, i.e., locally
isomorphic to a direct product of a group of polynomial growth and
a semisimple group.

The following several paragraphs contain omitted in  [ibid.]
details concerning existence of the exponential radical and
decomposition its Lie algebra in the general connected case.
First, we give a definition on the Lie algebra level.
\begin{df}
We say that a (real or complex) Lie algebra is
\emph{R-decomposed} if it is a direct sum of a semisimple algebra
and an  algebra of Type~R.
\end{df}
Obviously, $G$ is P-decomposed iff its Lie algebra $\fg$ is
R-decomposed. Note that both properties have alternative
descriptions: a connected Lie group is P-decomposed iff it has
the Rapid Decay Property \cite{CPS07} and a unimodular Lie
algebra  is R-decomposed iff it is a B-algebra in the sense of
Varopoulos \cite[Sect.~1.8]{Va96}.

The proof of the following lemma is straightforward.
\begin{lm}\label{Rdecl}
A  Lie algebra  is R-decomposed  iff any maximal semisimple
subalgebra without compact summands is complemented and each
complement is of Type~R.
\end{lm}

Suppose temporarily that $\fg$ is a real Lie algebra. Let $\fr$ be
the solvable radical of~$\fg$ and  $\fs$  a Levi complement. Write
$$
\fs=\fs_c\oplus\fs_{nc},
$$
where $\fs_c$ is compact and $\fs_{nc}$  is maximal semisimple
without compact summands. Denote by $\fe_r$  the ideal in $\fr$
s.t. $\fr/\fe_r$ is the largest quotient of $\fr$  that is
R-decomposed and set
\begin{equation}\label{defe}
\fe\!:=\fe_r+[\fs_{nc},\fr]\,.
\end{equation}

\begin{lm}\label{eisnilp}
The subspace $\fe$ is a nilpotent ideal in $\fg$.
\end{lm}
\begin{proof}
First, note that $[\fs_{nc},\fr]$ is an ideal in $\fg$ (see the
Lie group form  in \cite[Lem.~6.8]{Co08}). Since $\fe_r$  is an
ideal in $\fr$ that is stable under  automorphisms of $\fr$
(Theorem~\ref{GuOs}(B)), it is stable under derivations, so
$\fe_r$ is an ideal in $\fg$. Therefore  $\fe$ is an ideal
in~$\fg$.

Let $\fr_\infty$ be the intersection of the lower central series
of $\fr$. Since  $\fr/\fr_\infty$ is nilpotent, it is of Type~R;
so $\fe_r\subset \fr_\infty$. Note that $\fr_\infty\subset
[\fr,\fr] $, which is nilpotent. On the other hand, by \cite[Cor.
5.4.15]{HiNe}, we have that $[\fg,\fr]$ is a nilpotent ideal of
$\fg$.  So  $\fe_r$ and $[\fs_{nc},\fr]$ are both contained in
nilpotent ideals; hence they are nilpotent themselves. Therefore
$\fe$ is nilpotent.
\end{proof}

\begin{lm}\label{gelaqal}
The Lie algebra $\fg/\fe$ is the largest quotient of $\fg$ that
is R-decomposed.
\end{lm}
\begin{proof}
First, we claim that  for a finite family of ideals
$\fj_1,\ldots,\fj_n$ s.t. each $\fg/\fj_k$ is R-decomposed, all
algebras $\fg/(\cap \fj_k)$ is also R-decomposed (cf. \cite{Gu80}
for Type~R algebras). Indeed, note that the property to be
R-decomposed is inherited by finite sums and subalgebras but
$\fg/(\cap \fj_k)$ is isomorphic to the range of $\fg \to
\bigoplus_k \fg/\fj_k$.

Suppose that  $\fj$ is an ideal in $\fg$ s.t. $\fg/ \fj$ is
R-decomposed. To show that $\fj\subset \fe$ note that $\fg/\fr$ is
R-decomposed, so $\fg/(\fr\cap \fj)$ is R-decomposed. So we can
assume that $\fj\subset \fr$.  Further, $\fr+\fs_c$ is a
semidirect sum and  an ideal in~$\fg$.  Moreover,
$\fg=(\fr\rtimes\fs_c)\rtimes \fs_{nc}$. Then $\fg/\fj$ is
isomorphic to $(\fr/ \fj\rtimes\fs_c)\rtimes \fs_{nc}$. Therefore
$\fs_{nc}$ is   maximal in $\fg/\fj$ as a semisimple subalgebra
without compact summands. Since $\fg/\fj$ is R-decomposed,
Lemma~\ref{Rdecl} implies that $\fs_{nc}$  is complemented and
$\fr/ \fj\rtimes\fs_c$ is of Type~R. In particular,
$[\fs_{nc},\fr]\subset \fj $. Note that $\fr/ \fj$ is the radical
of $\fr/ \fj\rtimes\fs_c$; so $\fr/ \fj$  is of Type~R.  By the
definition of $\fe_\fr$, we have $\fe_\fr\subset\fj$.  It follows
from \eqref{defe} that $\fe\subset\fj$.

On the other hand, since $\fe\subset\fr$, we obtain
$\fg/\fe\cong(\fr/ \fe\rtimes\fs_c)\rtimes \fs_{nc}$.  At the same
time, $[\fs_{nc},\fr]\subset \fe $; therefore the action of
$\fs_{nc}$ on $\fr/ \fe\rtimes\fs_c$ is trivial. Thus
$\fg/\fe\cong(\fr/ \fe\rtimes\fs_c)\oplus \fs_{nc}$.  Since $\fr/
\fe$ and $\fs_c$ are of Type~R, it follows from
\cite[Pr.~6.28]{Pat} that $\fr/ \fe\rtimes\fs_c$ is of Type~R.
\end{proof}

The following corollary follows easily from Lemma~\ref{gelaqal}.
\begin{co}\label{Liealde}
For any connected Lie group $G$, the exponential radical  exists
and coincides with the closure of the integral subgroup
$\langle\exp \fe\rangle$, where $\fe$ is defined by~\eqref{defe}
for the Lie algebra $\fg$ of $G$.
\end{co}

Our next goal is to show that, for a complex Lie algebra $\fg$,
the decomposition~\eqref{defe} has the simplified  form
\begin{equation}\label{rinsr}
\fe=\fr_\infty+[\fs,\fr]\,,
\end{equation}
where $\fr_\infty$ denotes the intersection of the lower central
series of $\fr$. It is evident that in this case  $\fs_{nc}=\fs$
but the equality $\fe_\fr=\fr_\infty$ needs a little more work.

\begin{lm} \label{TRisN}
A complex Lie algebra is of  Type~R iff it is nilpotent.
\end{lm}
\begin{proof}
Suppose that $\fg$ is a complex Lie algebra is of  Type~R, i.e.,
the eigenvalues of $\ad \xi$ are contained in $i\R$ for all
$\xi\in\fg$. In particular, the eigenvalues of $\ad (i\xi)$ also are
contained in $i\R$. Therefore, each eigenvalue is $0$; so, by
Engel's Theorem, $\fg$ is nilpotent.

On the other hand, each nilpotent Lie algebra is of Type~R.
\end{proof}

\begin{co} \label{PGisN}
A connected complex Lie group is of polynomial growth iff it is nilpotent.
\end{co}

\begin{lm} \label{fefgin}
Let $\fg$ be a solvable complex Lie algebra and  $\fe$ a real
ideal s.t. $\fg/\fe$ is the largest Type~R quotient of $\fg$. Then
$\fe=\fg_\infty$.
\end{lm}
\begin{proof}
First, we show that $\fe$ is a complex ideal.  Note that $i\fe$ is
a real ideal. We claim that $\fg/i\fe$ is of  Type~R. Indeed, we
have to show that $[\xi,\eta]-\la\eta\in i\fe$ for some
$\xi,\eta\in\fg$ and $\la\in\CC$ implies $\la\in i\R$. Multiplying
by $i$, we obtain $[\xi,i\eta]-\la i\eta\in \fe$.    Since
$\fg/\fe$ is of  Type~R, we have that $\la\in i\R$.

Further, $\fg/\fe$ is the largest  Type~R quotient of $\fg$; so
$\fe \subset i\fe$. Therefore $\fe=i\fe$ and $\fe$ is a complex
ideal.

Since $\fg/\fe$ is a complex Lie algebra of Type~R,
Lemma~\ref{TRisN} implies that $\fg/\fe$ is nilpotent. On the
other hand, $\fg/\fg_\infty$ is the largest nilpotent quotient of
$\fg$. Therefore, $\fg_\infty\subset \fe$. Finally, since
$\fg/\fg_\infty$ is of Type~R, we have $\fe\subset \fg_\infty$.
\end{proof}

\begin{co}\label{corinsr}
If $\fg$ is a complex Lie algebra, then  \eqref{rinsr} is
satisfied.
\end{co}

Now we return to the Lie group level.
\begin{pr} \label{exraco}
The exponential radical of a connected complex Lie group $G$
coincides with the normal complex Lie subgroup $E$ s.t. $G/E$ is
the largest quotient of $G$ that is locally isomorphic to a direct
product of a nilpotent  and   semisimple complex Lie group.
\end{pr}
\begin{proof}
Let $\fg$ be the Lie algebra of $G$ and $\fe$ defined
by~\eqref{defe}. Corollary~\ref{corinsr} implies that
$\fe=\fr_\infty+[\fs,\fr]$; thus $\fe$ is a complex ideal
in~$\fg$. Since $E=\overline{\langle\exp \fe\rangle}$, it is a
complex Lie subgroup of $G$. Therefore $G/E$ is a complex Lie
group. It follows from Corollary~\ref{PGisN} that any P-decomposed
complex Lie group is locally isomorphic to a direct product of a
nilpotent and  a semisimple complex group. It is not hard to see
that $G/E$ is the largest quotient with this property.
\end{proof}

Now we consider the exponential radical of a connected linear
complex Lie group. Recall that every such group has the form
$B\rtimes L$, where $B$  is complex simply connected solvable and
$L$ is linearly complex reductive \cite[Th.~16.3.7]{HiNe}.

\begin{thm} \label{rexp1}
Let $G$ be a connected complex Lie group  of the form $B\rtimes
L$, where $B$  is complex simply connected solvable and $L$ is
linearly complex reductive. Suppose that the exponential radical
of $G$ is  $\{1\}$. Then the action of $L$ on $B$ is trivial, so
$G=B\times L$.
\end{thm}
\begin{proof}
Consider a Levi-Malcev decomposition $L=TS$, where $T$~is the
solvable radical of~$L$ and~$S$ is a semisimple Levi subgroup.
Since both $T$ and $S$ are integral subgroups, it suffice to
prove that the actions of the corresponding Lie subalgebras $\ft$
and $\fs$ on the Lie algebra $\fb$ of $B$ are trivial.

Obviously, $\fb+\ft$ is a semidirect sum. Since the Lie algebra
$\fl$ of $L$ is  reductive, $\ft$~is  central in~$\fl$ \cite[Pr.
5.7.3]{HiNe}. Then $\fb\rtimes\ft$ is a solvable ideal; hence,
$B\rtimes T$ is a closed normal solvable subgroup. Moreover, it
follows from Proposition~\ref{exraco} that the solvable radical
$R$ of $G$ is nilpotent. Therefore $B\rtimes T$ is nilpotent (in
fact,  $B\rtimes T=R$).

Since $T$ is linearly complex reductive,  $\fb$ is a completely
reducible module w.r.t. the adjoint action of $T$ [ibid.,
15.2.10]. On the other hand, since~$R$ connected and
nilpotent, the adjoint action of $R$ on its Lie algebra $\fr$ is
unipotent \cite[Ch.~2,  Th. 1.6]{OV3}. In particular, the action
of $T$ on $\fb$ is unipotent. Thus the action of $T$ on $\fb$ is
trivial, so is the action of $\ft$ on $\fb$.

On the other hand, Proposition~\ref{exraco} implies that the
action of $\fs$ on $\fr$ (in particular, on $\fb$) is trivial.
\end{proof}

\begin{pr}\label{expe}
Let $G$  be a connected linear complex Lie group. Then the exponential
radical of $G$ is simply connected nilpotent and coincides with $\exp \fe$.
\end{pr}
\begin{proof}
Since $\fe\subset [\fg,\fr]$ and $(G, R)$ is an integral
subgroup, the normal integral subgroup $\langle\exp \fe\rangle$
is contained in $(G, R)$ \cite[Lem.~11.2.2]{HiNe}.

Let $K$ be a maximal compact subgroup of $G$. Since $G$ is linear,
we have $K^*\cap(G, R)=\{1\}$ \cite[Th.~16.3.7]{HiNe}. Then the
normal subgroup $\langle\exp \fe\rangle$ intersects $K$ trivially;
hence, it intersects all maximal compact subgroups of $G$
trivially. Whence it is closed and simply connected
\cite[Cor.~14.5.6, Th.~14.3.11]{HiNe}. Thus $E=\langle\exp
\fe\rangle$. Finally, by Lemma~\ref{eisnilp}, $\fe$  is nilpotent;
therefore $\langle\exp \fe\rangle=\exp \fe$.
\end{proof}

Finally, we can prove the main result of the section.

\begin{thm}\label{GEdecom}
Let $G=B\rtimes L$, where $B$  is simply connected solvable and $L$ is
linearly complex reductive, and let  $E$ be the exponential radical of $G$.
Then $E\subset B$ and $G/E\cong  B/E\times L$.
\end{thm}
 \begin{proof}
Note that   $\fr_\infty\subset [\fr,\fr]\subset \fr\cap [\fg,\fg]$
and $[\fs,\fr] \subset \fr\cap  [\fg,\fg]$. Since, by
Corollary~\ref{corinsr},   $\fe=\fr_\infty+[\fs,\fr]$   and, by
\cite[Lem.~5.6.4(ii)]{HiNe}, $\rad[\fg,\fg]=\fr\cap[\fg,\fg]$, we
have $\fe\subset \rad[\fg,\fg]$.  Proposition~\ref{expe} implies
that $E\subset \Rad(G,G)$. It follows from \cite[Pr.~4.44]{Le02}
that the subgroup $B$ contains the representation radical, which
is, by definition, is the intersection of all kernels of
semisimple holomorphic representations of $G$.  The assumption of
the theorem yields  that that $G$ is linear
\cite[Th.~16.3.7]{HiNe}, so the representation radical of $G$
coincides with $\Rad(G,G)$ \cite[Cor.~4.39]{Le02}. Hence,
$E\subset B$.

Consider the action of $L$ on $B/E$ s.t. $G/E\cong B/E\rtimes L$.
Since $G$ is linear, it follows from Proposition~\ref{expe} that
$E$ is simply connected. Therefore, $B/E$ is simply connected and,
evidently, solvable. On the other hand, by
Proposition~\ref{exraco}, the exponential radical of $G/E$ is
trivial. Thus Theorem~\ref{rexp1} implies that $G/E\cong B/E\times
L$.
\end{proof}

\section{Asymptotic behavior of a word length function}\label{ABWLF}

Recall that a \emph{length function} on a locally compact group
$G$ is a locally bounded function $\ell\!:G\to \R$ s.t.
$$
 \ell(gh)\le \ell(g)+\ell(h)\qquad (g, h \in G)\,.
$$
Note that we does not assume in general that $\ell$ is
\emph{symmetric}, i.e., $\ell(e) = 0$ and $\ell(g^{-1})=\ell(g)$
for all $g\in G$.

If $G$ is compactly generated, i.e., there is a relatively
compact generating set $U$ ($\bigcup_{n=0}^{\infty} U^{n} = G$,
where  $e\in U$ and $U^{0}\!:=\{e\}$), then the function defined
by
\begin{equation}\label{wordlen}
\ell_U(g)\!: = \min \{ n \!: \, g \in U^{n} \}\,.
\end{equation}
is a length function. Any such length function is called  a
\emph{word length function}.

For  positive functions  $\tau_1$ and $\tau_2$ on a set $X$, we
say that $\tau_1$   \emph{dominated by}  $\tau_2$ (at infinity)
and write $\tau_1\lesssim\tau_2$ if there are $C,D>0$ s.t.
$$\tau_1(x)\le C\tau_2(x) + D\qquad (x\in X)\,.$$
Moreover,  if $\tau_1\lesssim\tau_2$ and $\tau_2\lesssim\tau_1$,
then we say that  $\tau_1$ and $\tau_2$  are \emph{equivalent} (at
infinity) and write $\tau_1\simeq\tau_2$. Note that the length
functions equivalence means that we have a bijective
quasi-isometry between the corresponding metric spaces.

This section is devoted to the proof of the following result.

\begin{thm}\label{qitiG}
Suppose that $G$ is a connected linear complex Lie group. Fix a
decomposition $G=B\rtimes L$, where $B$  is complex simply
connected solvable and $L$ is linearly complex reductive. Let $E$
be the exponential radical of $G$, $\pi\!:B\to B/E$ the quotient
homomorphism, $\fb$ and $\fe$ the Lie algebras of $B$ and $E$,
resp., and $\fv$  a complementary subspace to $\fh\cap \fe  $
in~$\fh$, where $\fh$ is a Cartan subalgebra in $\fb$.   Then
$$
\tau\!:\fe\times \fv\times L\to G\!:(\eta,\xi,l)\mapsto
\exp(\eta)\exp(\xi)\,l
$$
is a  biholomorphic equivalence of complex manifolds s.t.
\begin{equation}\label{thredec}
\ell(\exp(\eta)\exp(\xi)l)\simeq
\log(1+\ell_0(\exp\eta))+\ell_1(\pi(\exp(\xi)))+\ell_2(l)
\end{equation}
on $\fe\times \fv\times L$, where $\eta\in \fe$, $\xi\in\fv$, and
$l\in L$,   and
$$
\log(1+\ell_0(\exp\eta))\simeq \log (1+\|\eta\|)\quad \text{on
$\fe$}\,,
$$
where $\|\cdot\|$ is a norm on $\fe$ and $\ell$, $\ell_0$,
$\ell_1$, and $\ell_2$ are word length functions on $G$, $E$,
$B/E$, and $L$, resp.
\end{thm}

We begin with a decomposition of a word length function on a
semidirect product.

\begin{pr}\label{ellNH}
Let $G=N\rtimes H$ be a semidirect product of compactly generated
locally compact groups. Suppose that $\ell$ and $\ell_1$  are
word length functions on $G$ and $H$, resp.  Then
$$
\ell(nh)\simeq  \ell(n)+\ell_1(h)\quad\text{on $N\times H$}
\quad(\text{$n\in N$, $h\in H$})\,.
$$
\end{pr}

For the proof, we need three lemmas.

\begin{lm}\label{wlfeq} \emph{(\cite[Th.~1.1.21]{Sch93} or
\cite[Th.~4.3(a)]{Ak08})} Each (in particular, each symmetric)
word length function dominates all length functions. As a
corollary,  all word length functions are equivalent.
\end{lm}

\begin{lm}\label{GG1}
Let $\pi\!: G\to G_1$ be a continuous homomorphism of compactly
generated locally compact groups and let $\ell$ and $\ell_1$ be
word length functions on $G$ and $G_1$, resp.

\emph{(A)} Then $ \ell_1(\pi(g))\lesssim  \ell(g)$ on $G$.

\emph{(B)} If, in addition,  $\si\!:G_1\to G$  is a continuous
homomorphism s.t. $\pi\si=1$, then $\ell_1(g_1)\simeq
\ell(\si(g_1))$ on $G_1$.
\end{lm}
\begin{proof}
By Lemma~\ref{wlfeq}, we can choose relatively compact generating
sets $U$ and $U_1$ determining  $\ell$ and $\ell_1$, resp., at our
request. If we put $\pi(U)\subset U_1$, then $ \ell_1(\pi(g))\le
\ell(g)$ for each $g\in G$, which proves part~(A). If we put
$\si(U_1)\subset U$, then $ \ell(\si(g_1))\le \ell_1(g_1)$ for
each $g_1\in G_1$. Thus part~(B) follows part~(A).
\end{proof}

\begin{lm}\label{NXdec}
\emph{(cf. \cite[Lem.~4.3]{Co11})} Let $G$ be a  group and  $N$ a
normal subgroup. Suppose that there are symmetric length functions
$\ell$ and $\ell_1$ on $G$ and $G/N$, resp., and a subset $X$ in
$G$ s.t.
$$
\ell_1(\pi(x))\simeq \ell(x)\quad\text{on
$X$}\quad\text{and}\qquad\ell_1(\pi(g))\lesssim
\ell(g)\quad\text{on $G$}\,,
$$
where $\pi\!:G\to G/N$ is the quotient homomorphism. Then
$$
\ell(nx)\simeq  \ell(n)+\ell_1(\pi(x))\quad\text{on $N\times X$}
\quad(\text{$n\in N$, $x\in X$})\,.
$$
\end{lm}
\begin{proof}
Since $\ell(x)\lesssim \ell_1(\pi(x))$ on $X$, we have
$$
\ell(nx)\le  \ell(n)+\ell(x)\lesssim  \ell(n)+\ell_1(\pi(x))\,.
$$
On the other hand, $\ell(n)\le
\ell(nx)+\ell(x^{-1})=\ell(nx)+\ell(x)$ and
$\ell_1(\pi(x))=\ell_1(\pi(nx))\lesssim\ell(nx)$.  Therefore,
\begin{multline*}
\ell(n)+\ell_1(\pi(x))\le \ell(nx)+\ell(x)+\ell_1(\pi(x))\lesssim \\
\ell(nx)+2\ell_1(\pi(x))\lesssim  3\ell(nx)\simeq \ell(nx)\,.
\end{multline*}
\end{proof}

\begin{proof}[Proof of Proposition~\ref{ellNH}]
Consider a homomorphism $\si\!:H\to G$ splitting the quotient
homomorphism $\pi\!:G\to H$. Lemma~\ref{GG1} implies that
$\ell_1(h)\simeq \ell(\si(h))$ on~$H$. Since $\ell_1(\pi(g))$ is
a length function on $G$, by virtue of Lemma~\ref{wlfeq}, we have
$\ell_1(\pi(g))\lesssim  \ell(g)$ on $G$. Evidently, we can
assume that $\ell$ and $\ell_1$ are symmetric and, finally, apply
Lemma~\ref{NXdec} with $X=\si(H)$.
\end{proof}

Further, we need asymptotic behavior of the restriction of a word
length function on the exponential radical.

\begin{pr}\label{ellnorm}
Let $G$ be a connected linear complex Lie group,  $E$  the
exponential radical of $G$, and $\fe$  the Lie algebra of $E$. If
$\|\cdot\|$ is a norm on $\fe$ and $\ell$ is a word  length
function on $G$, then
$$
\ell(\exp(\eta)) \simeq \log (1+\|\eta\|) \quad\text{on $\fe$} \,.
$$
\end{pr}
For the proof, we need the following lemma.
\begin{lm}\label{loglogeq}
Let $G$ be a simply connected nilpotent (real or complex) Lie
group with the Lie algebra  $\fg$. If $\|\cdot\|$ is a norm on
$\fg$ and $\ell$ is a word length function on $G$ then
 $$
\log(1+\|\xi\|)\simeq \log(1+\ell(\exp\xi))\qquad (\xi\in \fg)\,.
$$
\end{lm}
\begin{proof}
It is sufficient to show that  $\ell(\exp\xi)\lesssim
\|\xi\|\lesssim \ell(\exp\xi)^{k}$  for some constant $k\ge 1$.
But this estimate is well known and follows, e.g., from
\cite[Lem.~II.1]{Gu73} or \cite[(4.2)]{Kar94}.
\end{proof}

Recall that  a closed subgroup $H$ with  word  length function
$\ell_0$  is said to be \emph{strictly exponentially distorted} in
a locally compact group $G$  with  word length function $\ell$ if
\begin{equation}\label{sedustde}
\log (1+\ell_0(h))\simeq \ell(h) \quad\text{on $H$} \,.
\end{equation}

\begin{proof}[Proof of Proposition~\ref{ellnorm}]
Let $\ell_0$ be a word  length function on $E$. Since the
exponential radical is strictly exponentially distorted
\cite[Th.~6.5]{Co08}, we obtain
$$
\log (1+\ell_0(g))\simeq \ell(g) \quad\text{on $E$}
\quad(\text{$g\in E$})\,.
$$
At the same time, by Lemma~\ref{loglogeq},
$$
\log(1+\ell_0(\exp\eta))\simeq \log(1+\|\eta\|)\qquad (\eta\in
\fe)\,.
$$
\end{proof}

Let $E$, $B$ and $L$ be as in Theorem~\ref{GEdecom}. Note that
$B\to B/E$  (and $G\to G/E$) can be non-split; so
Proposition~\ref{ellNH} is not applicable in the direct way. To
find a decomposition of a word length function on $G$ in the
general case we use the following trick from \cite{Co08}. Consider
a Cartan subalgebra $\fh$ in $\fb$; then, in particular, $\fh$ is
nilpotent and $\fh +\fr_\infty = \fb$ \cite[Ch.~7, \S~2, N.~1,
Cor.~3]{Bou2}. Note that  $\fh +\fe  = \fb$ and choose a
complementary subspace $\fv$ of $\fh\cap \fe  $ in~$\fh$.

The following proposition is a variant of manifold splitting.  It
can be proved as in \cite[Lem.~14.3.6]{HiNe} but succeeding proof
uses the Product Formula.

\begin{pr}\label{expVE}
The map
 $$\tau\!:\fe\times\fv\to B\!:(\eta,\xi)\mapsto \exp(\eta)\exp(\xi) $$
is a biholomorphic equivalence of complex manifolds.
\end{pr}
\begin{proof}
First, let us prove that $\tau$ is surjective. Since $B$  is
simply connected solvable, the subgroup $H\!:=\langle
\exp\fh\rangle $ is simply connected and closed
\cite[Prop.~11.2.15]{HiNe}. Further, $E$ is normal in $B$; so the
subgroup $EH$ is a subgroup in $B$.  The equality $\fh +\fe  =
\fb$ implies that $EH$ is dense in $B$. Applying
\cite[Prop.~11.2.15]{HiNe} again, we get that $EH$ is closed.
Therefore  $B=EH$.

It follows from Proposition~\ref{expe} that $E$ is simply
connected nilpotent.  The subgroup $H$ is also simply connected
nilpotent. Therefore, for any element of $B$, there is a
decomposition $\exp(\eta)\exp(\ze)$, where $\eta\in \fe$ and
$\ze\in \fh$. Write $\ze=\nu+\xi$, where $\nu\in \fh\cap \fe$ and
$\xi\in \fv$.  Then, by the Product Formula \cite[Pr.
9.2.14(1)]{HiNe},
$$
\exp(\ze)=\lim_{n\to \infty} (\exp(\nu/n)\exp(\xi/n))^n\,.
$$
The claim is that for each $n\in\N$ there is $g\in E$ s.t.
$$(\exp(\nu/n)\exp(\xi/n))^n=g\exp (\xi)\,.$$ We proceed by
induction. If $n=1$, then the claim is obvious. Suppose that it is
true for $n-1$. Set $\nu'\!:=(n-1)\nu/n$ and $\xi'\!:=(n-1)\xi/n$
and write
$$
\Bigl(\exp\frac{\nu'}{n-1}\exp\frac{\xi'}{n-1}\Bigr)^{n-1}=g'\exp
(\xi')
$$
for some $g'\in E$. Then
\begin{align*}
(\exp(\nu/n)&\exp(\xi/n))^n=(\exp(\nu/n)\exp(\xi/n))^{n-1}\exp(\nu/n)\exp(\xi/n)=\\
&(\exp(\nu'/(n-1))\exp(\xi'/(n-1)))^{n-1}\exp(\nu/n)\exp(\xi/n)=\\
&g'\exp(\xi')\exp(\nu/n)\exp(\xi/n)= g\exp (\xi)\,,
\end{align*}
where $g\!:=g'\exp(\xi')\exp(\nu/n)\exp(-\xi')$ is in $E$.  The claim is
proved.

Thus $\exp(\ze)=\lim_{n\to \infty}g_n \exp(\xi)$, where $g_n\in
E$. Since $E$ is closed,  $\exp(\ze)$ is in the range of $\tau$,
so is $\exp(\eta)\exp(\ze)$. Hence $\tau$ is surjective.

Secondly, prove that $\tau$ is injective.  Suppose  that
$\tau(\eta,\xi)=\tau(\eta',\xi')$ for some  $\eta,\eta'\in \fe$
and $\xi,\xi'\in \fv$. Let $p\!:\fb\to \fb/\fe$ and $\pi\!:B\to
B/E$ be the quotient maps. The exponential map is natural; so
$$
\exp_{B/E}(p(\xi))=
\pi\tau(\eta,\xi)=\pi\tau(\eta',\xi')=\exp_{B/E}(p(\xi'))\,.
$$
Since  $\fb/\fe$ is nilpotent, $\exp_{B/E}$ is bijective.
Therefore $p(\xi)=p(\xi')$. Finally, note that $p$ is injective on
$\mathfrak{v}$, hence $\xi=\xi'$.  Thus $\exp(\eta)=\exp(\eta')$.
Since the exponential map is natural and $E\to B$ is injective,
$\exp_E(\eta)=\exp_E(\eta')$. But  $\fe$ is nilpotent, whence
$\exp_{E}$ is bijective; so $\eta=\eta'$.

Finally,  $\fe\times\fv$ and $B$ are both $\CC^m$ for some $m$.
So we can treat $\tau$ as a holomorphic injection  $\CC^m\to
\CC^m$. Then, by \cite[Th.~8.5]{KaKa}, $\tau$ is a biholomorphic
equivalence onto the range of $\tau$, which coincides with
$\CC^m$. This concludes the proof.
\end{proof}

\begin{lm}\label{ellell1}
Let $\ell$ and $\ell_1$ be word length functions  on $G$ and
$G/E$, resp. Then $\ell(g)\simeq \ell_1(\pi(g))$ on $\exp\fv$.
\end{lm}
\begin{proof}
First, note  that $\ell_1(\pi(g))\lesssim \ell(g)$ on $G$ by
Lemma~\ref{GG1}(A).

To prove $\ell(g)\lesssim  \ell_1(\pi(g))$ on $\exp\fv$ we
follow~\cite[Lem.~5.2]{Co08} with small modifications. Let $\pi$
denote the projection $G\to G/E$. Fix a compact symmetric
generating subset $S$ of $H\!:=\exp\fh$ and denote by $\ell_2$
the corresponding word length function  on $H$.  We can assume
that $\ell_1$ is the word length function on $G/E$ corresponding
to~$\pi(S)$.

If $g\in H$, then write $\pi(g)$ as an element of minimal length
w.r.t. $\pi(S)$, i.e., $\pi(g)=\pi(s_1)\dots\pi(s_m)$, where
$s_1,\ldots, s_m\in S$. Set $h=s_1\cdots s_m$.  We can assume that
$\ell$ is a word length function corresponding to a compact
generating set containing $S$. Then $\ell(h)\le m=\ell_1(\pi(g))$.
As $h^{-1}g$ belongs to the exponential radical $E$, which is
strictly exponentially distorted by \cite[Th.~1.1(3)]{Os02}, we
get that $\ell(h^{-1}g)\lesssim \log(1+\ell_0(h^{-1}g))$ for $g\in
H$, where $\ell_0$ is a word length function on $E$. (Here we
consider $h$ as a function of $g$) Therefore,
\begin{equation}\label{ell10}
\ell(g)\le \ell(h)+\ell(h^{-1}g)\lesssim
\ell_1(\pi(g))+\log(1+\ell_0(h^{-1}g))\quad(g\in H)\,.
\end{equation}

Write $g=\exp \xi$ and $h=\exp \eta$ for $\xi,\eta\in \fh$.  Fix a
norm $\|\cdot\|$ on $\fb$, the Lie algebra of $B$. Since $E$ is
nilpotent, we can use Lemma~\ref{loglogeq} for the restriction
$\|\cdot\|$ on $\fe$ and $\ell_0$. Hence,
$$
\log(1+\ell_0(\exp(-\eta)\exp \xi))\simeq \log(1+\|\exp^{-1}(\exp
(-\eta)\exp \xi)\|)\,.
$$
(Here we consider $\eta$ as a function of $\xi$.) Since $H$ is
nilpotent,  its group law is given by a polynomial and we have an
upper bound
$$
\|\exp^{-1}(\exp (-\eta)\exp \xi)\|\le A(1+\|\xi\|)^k(1+\| \eta\|)^k
$$
for some constants $A,k\ge 1$. This implies
$$
\log(1+\ell_0(\exp(-\eta)\exp \xi)))\lesssim
\log(1+\|\eta\|)+\log(1+\| \xi\|)\,.
$$
Further,
$$\ell_2(h)\le m=\ell_1(\pi(g))\le  \ell(g)\le
\ell_2(g)$$
 (since $\ell_2$ is a word length function on $H$, it
dominates each length function). As $H$ is nilpotent, we can
apply Lemma~\ref{loglogeq} for the restriction $\|\cdot\|$ on
$\fh$ and $\ell_2$  and get
$$
\log(1+\|\eta\|)\lesssim\log(1+\ell_2(h))\lesssim
\log(1+\ell_2(g))\lesssim \log(1+\|\xi\|)\,.
$$
Therefore we have from \eqref{ell10} that  on $H$
$$
\ell(g)\lesssim \ell_1(\pi(g))+\log(1+\|\xi\|)\,.
$$

Now suppose that $g=\exp \xi$ for some $\xi\in \fv$. If
$\|\cdot\|'$ is a norm on $\fb/\fe$ and $p\!:\fg\to \fg/\fe$ is
the quotient map, then $\|p(\xi)\|'\simeq \|\xi\| $ on $\fv$.
Since $\exp (p(\xi))= \pi(\exp\xi)$, we have from the application
of Lemma~\ref{loglogeq} for $\|\cdot\|'$  and $\ell_1$ that
$$
\log(1+\|\xi\|)\simeq  \log(1+\|p(\xi)\|')\lesssim
\log(1+\ell_1(\pi(g)))\lesssim \ell_1(\pi(g))\,.
$$
Thus $\ell(g)\lesssim \ell_1(\pi(g))$ on $\exp\fv$.
\end{proof}

We are now in a position to prove the decomposition result.

\begin{proof}[Proof of Theorem~\ref{qitiG}]
It follows from Proposition~\ref{expVE}  that $\fe\times \fv\to
B$ and $\fe\times \fv\times L\to G$ are biholomorphic
equivalences.

Consider the quotient map $\si\!:G\to G/E$ and the length
function $\wt\ell(h,l)\!:=\ell_1(h)+\ell_2(l)$ on $B/E\times L$.
According to Theorem~\ref{GEdecom}, we can identify $G/E$ with
$B/E\times L$ and $\si$ with $\pi\times 1\!: B\rtimes L\to
B/E\times L$. Proposition~\ref{ellnorm} implies that it
sufficient to show that
\begin{equation}\label{ellwtell}
\ell(\exp(\eta)\exp(\xi)l)\simeq
\ell(\exp(\eta))+\wt\ell(\si(\exp(\xi)l))\quad \text{on
$\fe\times \fv\times L$}\,.
\end{equation}

By Lemma~\ref{ellell1}, we have $\ell(\exp(\xi))\simeq
\ell_1(\pi(\exp(\xi)))$ on~$\fv$. Besides, Proposition~\ref{ellNH}
yields that $\ell(bl)\simeq \ell(b)+\ell_2(l)$ on $B\times L$.
Combining these relations, we get $\ell(\exp(\xi)l)\simeq
\wt\ell(\si(\exp(\xi)l))$ on $\fv\times L$. On the other hand,
each length function is dominated by a word length function, so
$\wt\ell(\si(g))\lesssim \ell(g)$ on~$G$. Thus both conditions of
Lemma~\ref{NXdec} are satisfied for $\ell$ and $\wt\ell$ with
$X=(\exp\fv)L$; so application of this lemma completes the proof
of~\eqref{ellwtell}.
\end{proof}

\section{Holomorphic functions of exponential type}\label{HFET}
Recall that a \emph{submultiplicative weight} on a locally
compact group $G$ is a non-negative locally bounded function
$\om\!: G \to \R$ s.t.
$$
\om(gh)\le \om(g)\om(h)\qquad (g, h \in G)\,.
$$
Akbarov proposed the term 'semicharacter' in \cite{Ak08} but we
follow \cite{ArAMN}.

A~holomorphic function $f$ on a complex Lie group $G$ is said to
be of \emph{exponential type}, if there is a submultiplicative
weight $\om$ satisfying  $|f(g)|\le \om(g)$ for all $g\in G $. The
linear space (in fact, a locally convex algebra and even a
topological Hopf algebra) of all holomorphic function $f$ of
exponential type on $G$ is denoted by $\cO_{exp}(G)$
\cite[Sect.~5.3.1]{Ak08}.

Consider the Fr\'{e}chet space $\cO(G)$ of holomorphic functions
on a complex Lie group $G$ and its strong dual space
$\cA(G)\!:=\cO(G)'$ endowed with the convolution multiplication.
In fact, $\cA(G)$ is a $\Ptens$-algebra w.r.t. the convolution,
i.e., it is a complete Hausdorff locally convex topological
algebra with jointly continuous multiplication; it is called the
\emph{algebra of analytic functionals} on $G$ \cite{Lit}.

If $A$ is a unital Banach algebra and $\pi\!:G\to \GL(A)$ is a
holomorphic homomorphism, then $\pi$ uniquely extends to a unital
continuous homomorphism $\bar\pi\!: {\cA}(G)\to A$ s.t.
\begin{equation}\label{pixg2}
\langle x, \bar\pi(a')\rangle =\langle a', \pi_x\rangle\qquad(a'
\in {\cA}(G),\,x \in A'),
\end{equation}
where  $\pi_x\in \mathcal{O}(G)$ is defined by
\begin{equation}\label{pixg4}
\pi_x(g)\!:= \langle x, \pi(g)\rangle\qquad(g\in G).
\end{equation}
On the other hand, for a unital continuous homomorphism
$\bar\pi\!: {\cA}(G)\to A$,
\begin{equation}\label{pixg3}
\pi(g)\!:= \bar\pi(\de_g)
 \qquad (g\in G)
\end{equation}
defines a holomorphic homomorphism $\pi\!:G\to \GL(A)$,
satisfying~(\ref{pixg2}) \cite{Lit}.

\begin{pr}\label{exthomo}
Let $f$ be a function  be a complex Lie group $G$. T.F.A.E.

\emph{(1)} $f\in \mathcal{O}_{exp}(G)$.

\emph{(2)} There exist a unital Banach algebra $A$, a  holomorphic
homomorphism $\pi\!:G\to \GL(A)$, and $x\in A'$ s.t.
\begin{equation}\label{fgzpig}
f(g)\!:= \langle x, \pi(g)\rangle \qquad(g\in G).
\end{equation}

\emph{(3)} $f$  is a coefficient of a holomorphic representation
in some Banach space.
\end{pr}

Recall that an \emph{Arens-Michael envelope} of a $\Ptens$-algebra
$A$ is  a pair $(\widehat A, \io_A)$, where $\widehat A$ is an
Arens-Michael algebra  and $\io_A$ is a continuous homomorphism $A
\to \widehat A$ s.t. for any Arens-Michael algebra $B$ and for
each continuous homomorphism $\phi\!: A \to B$ there exists a
unique continuous homomorphism $\widehat\phi\!:\widehat A \to B$
with $\phi=\wh\phi\io_A$ \cite[Chap.~5]{X2}.

\begin{proof}
(1)$\Rightarrow$(2). Suppose that $f\in \mathcal{O}_{exp}(G)$. Put
$\cA_{exp}(G)\!:=\mathcal{O}_{exp}(G)'$ and consider $f$ as a
functional on $\cA_{exp}(G)$. The natural map $\nu\!:\cA(G)\to
\cA_{exp}(G)$ is an Arens-Michael envelope. (This result is proved
in \cite[Th.~5.2]{Ak08} for Stein groups but  the argument in the
general case is similar.) Therefore there exists a continuous
submultiplicative prenorm $\|\cdot\|$ on $\cA(G)$ and $C>0$ s.t.
$|\langle f, \nu(a')\rangle| \le C\, \|a'\|$  for all
$a'\in\cA(G)$. Denote by $A$ the Banach algebra that is the
completion of $\cA(G)$ w.r.t. $\|\cdot\|$ and by $\bar\pi$ the
corresponding continuous homomorphism $\cA(G)\to A$. Then the
functional $f$ factors on some $x\in A'$ s.t. $\langle f,
\nu(a')\rangle=\langle x, \bar\pi(a')\rangle$ for all   $a'$.
Consider a  holomorphic homomorphism $\pi\!:G\to \GL(A)$  defined
by (\ref{pixg3}). Since $\langle f, \de_g\rangle=f(g)$, we get
$f(g)=\langle x, \bar\pi(\de_g)\rangle=\langle x, \pi(g)\rangle$.

(2)$\Rightarrow$(1). Suppose that  $\pi\!:G\to \GL(A)$ is a
holomorphic homomorphism  and $x\in A'$. Any function $f$ of the
form~(\ref{fgzpig}), being a composition of a holomorphic and
linear map,  is holomorphic.  Obviously, the function $g\mapsto
\|\pi(g)\|$, where $\|\cdot\|$ is the norm on $A$, is
submultiplicative and continuous; therefore it is a
submultiplicative weight; moreover,
$$
g\mapsto \max\{\|x\|,\,1\}\, \|\pi(g)\|
$$
is a submultiplicative weight. Since $|f(g)|\le \|x\|\,\|\pi(g)\|$
for all $g$, we have that $f$ is of exponential type.

(2)$\Leftrightarrow$(3). It is sufficient to note that  a function
of the form~(\ref{fgzpig}) is a coefficient of the representation
that is a composition of $\pi$ and the regular representation $A$
on itself.
\end{proof}

Recall that $\cO_{exp}(G)$ is endowed with an inductive locally convex
topology via  identification
$$
\cO_{exp}(G)=\varinjlim_{\om} \cO_\om(G)\,,
$$
where $\om$ runs all submultiplicative weights on $G$ and
$\cO_\om(G)$ is a Banach space defined by
\begin{equation} \label{cOom}
\cO_\om(G)\!:=\Bigl\{ f\in\cO(G) \!: |f|_\om\!:=\sup_{g\in
G}{\om(g)}^{-1}{|f(g)|}<\infty\Bigr\}\,.
\end{equation}
Note that this definition can be applied also for any complex
manifold and any locally bounded function  with values in
$[1,+\infty)$.

It is proved in \cite[Th.~4.5]{Ak08} that $\cO_{exp}(G)$ is a
projective stereotype algebra (at least, for a compactly generated
Stein group $G$). But $\cO_{exp}(G)$ is also an algebra in more
traditional category of functional analysis as it is seen from the
following lemma.

\begin{lm}\label{Oexpptnal}
If $G$ is a complex Lie group, then $\cO_{exp}(G)$ is a
$\Ptens$-algebra w.r.t. the point-wise multiplication.
\end{lm}
\begin{proof}
Since a product of two submultiplicative weights is a
submultiplicative weight, $\cO_{exp}(G)$ is an algebra.  Since
the strong dual of  Fr\'echet space is complete and
$\cO_{exp}(G)$ is the strong dual of the Fr\'echet space
$\cA_{exp}(G)$ \cite[Pr.~2.12]{ArAMN}, we get that $\cO_{exp}(G)$
is complete.

It remains to show  that the multiplication is jointly
continuous. Note that $\cO_{exp}(G)$  is endowed with the
inductive topology, i.e., the family of all absolutely convex
subsets $U$ in $\cO_{exp}(G)$  s.t. $U\cap \cO_\om(G)$ is open
for each submultiplicative weight $\om$ on $G$ is a base of
neighbourhoods of $0$.

Let $U$ be an open subset in $\cO_{exp}(G)$. Then for each
submultiplicative weight $\om$ there is $C_\om\ge0$ s.t.
$U_\om\!:=\{|f(g)|<C_\om \om(g)\,\forall \,g\in G\}$ is contained
in $U\cap \cO_\om(G)$. It is easy to see that $\bigcup_\om U_\om$
is an open subset in  $\cO_{exp}(G)$ and contained in $U$. For
each submultiplicative weight $\om$, the function
$g\mapsto\om(g)^{1/2}$ is also a submultiplicative weight. Set
$V_\om\!:=\{|f(g)|<C_\om^{1/2} \om(g)^{1/2}\,\forall \,g\in G\}$;
then  $\bigcup_\om V_\om$ is  open in  $\cO_{exp}(G)$. It is
obvious that $f_1,f_2\in V_\om$ implies $f_1f_2\in U_\om$. Thus
the multiplication is jointly continuous.
\end{proof}

It is not hard to see that any holomorphic homomorphism
$\phi\!:G\to H$  of complex Lie group induces a  $\Ptens$-algebra
homomorphism $\widetilde\phi\!:\cO_{exp}(H)\to \cO_{exp}(G)$
given by $[\widetilde\phi(f)](g)\!:=f(\phi(g))$.  Since the dual
map $\wt\phi'\!:\cA_{exp}(G)\to \cA_{exp}(H)$ coincides with the
image of  the homomorphism $\cA(G)\to \cA(H)$ under the
Arens-Michael envelope functor, $\wt\phi'$ is also a
$\Ptens$-algebra homomorphism.

Now we can prove that it suffice to study holomorphic functions
of exponential type only on linear groups.

\begin{thm}\label{redliexpf}
Let $G$ be a connected complex Lie group and let $\si\!:G\to
G/\LinC(G)$ be the quotient homomorphism. Then

{\em (A)}  $\wt\si'\!:\cA_{exp}(G)\to \cA_{exp}(G/\LinC(G))  $ is
a  $\Ptens$-algebra isomorphism.

{\em (B)}  $\wt\si\!:\cO_{exp}(G/\LinC(G))\to \cO_{exp}(G) $ is
a  $\Ptens$-algebra isomorphism.

{\em (C)}   Each functions in  $\cO_{exp}(G)$ is constant on
cosets of $\LinC(G)$.
\end{thm}
\begin{proof}
(A) It follows from Theorem~\ref{holhomline} that each holomorphic
homomorphism $G\to\GL(A)$ factors on~$\si$.  On the other hand,
\eqref{pixg2} and \eqref{pixg3} give a bijection between the set
of  holomorphic homomorphisms from~$G$ to $\GL(A)$ and the set of
unital continuous homomorphisms from $\cA(G)$ to~$A$. The same, of
course, is true for  $G/\LinC(G)$. Therefore each unital
continuous homomorphism $\cA(G)\to  A$, where $A$~is a unital
Banach algebra, factors  on $\cA(G)\to   \cA(G/\LinC(G))$. Since
$\cA(G)\to \cA_{exp}(G)$ and $\cA(G/\LinC(G))\to
\cA_{exp}(G/\LinC(G))$ are Arens-Michael envelopes
\cite[Th.~6.2]{Ak08}, we see $\wt\si'$ is a topological
isomorphism.

(B) It is sufficient to note that the strong dual of $\wt\si'$
coincides with $\si$.

Part~(C) follows (B) immediately.
\end{proof}

\begin{rem}
Nevertheless, even for simplest examples there are functions of
exponential type that are not coefficients of finite-dimensional
holomorphic representations. For example, if $G=\CC$, then any
finite-dimensional holomorphic representation has the form
$z\mapsto\exp(zT)$ for some complex matrix $T$. It is nor hard to
see from the Jordan decomposition of $T$ that all matrix
coefficients belongs to the algebra generated by $z$ and
$\{e^{\la z}:\,\la\in\CC\}$. This algebra is smaller than
$\cO_{exp}(\CC)$.
\end{rem}

For a complex manifold $M$ and   a locally bounded function
$\up\!:M\to [1,+\infty)$, denote by  $V_\up$  the closure of
$$
\acv\{\up(x)^{-1}\de_x:\,x\in M\}
$$
in $\cA(M)\!:=\cO(M)'$, where $\acv$ denotes the absolutely convex
hull.  Let $\|\cdot\|_{\up}$ be the Minkowski functional of
$V_\up$ and  ${\cA}_\up(M)$  the completion of ${\cA}(M)$ w.r.t.
$\|\cdot\|_{\up}$. Also, denote by ${\cA}_{\up^\infty}(M)$ the
completion of ${\cA}(M)$ w.r.t. the sequence of prenorms
$(\|\cdot\|_{\up^n};\,n\in\N)$, where $\up^n(x)\!:=\up(x)^n$.

Also, we put $\cO_{\up^\infty}(M)\!:=\bigcup_{n\in\N}
\cO_{\up^n}(M)$ and consider $\cO_{\up^\infty}(M)$ with the
inductive limit topology.

We denote by $E\Ptens F$ the complete projective tensor product
of locally convex spaces $E$ and $F$.

\begin{pr} \label{Aom12}
Let $M_1$ and $M_2$ be complex manifolds and let $\up_1\!:M_1\to
[1,+\infty)$ and $\up_2\!:M_2\to [1,+\infty)$ be locally bounded
functions. Set $\up(x_1,x_2)\!:=\up_1(x_1)\up_2(x_2)$.

\emph{(A)} Then the natural map
$\rho\!:\cA(M_1)\otimes\cA(M_2)\to\cA(M_1\times M_2)$ induces the
topological isomorphism of Banach spaces
$$
\cA_{\up_1}(M_1)\Ptens \cA_{\up_2}(M_2)\cong\cA_{\up}(M_1\times
M_2) \,.
$$
and the topological isomorphism of Fr\'echet spaces
$$
\cA_{\up_1^\infty}(M_1)\Ptens
\cA_{\up_2^\infty}(M_2)\cong\cA_{\up^\infty}(M_1\times M_2) \,.
$$

\emph{(B)} If, in addition, each $\cA_{\up_i^\infty}(M_i)$ is
nuclear  ($i=1,2$), then the natural map
$\cO(M_1)\otimes\cO(M_2)\to\cO(M_1\times M_2)$ induces the
topological isomorphism of locally convex spaces
$$
\cO_{\up_1^\infty}(M_1)\Ptens \cO_{\up_2^\infty}(M_2) \cong
\cO_{\up^\infty}(M_1\times M_2)\,.
$$
\end{pr}
\begin{proof}
(A) Since $\|\cdot\|_{\up_i}$ is the Minkowski functional of
$V_{\up_i}$ ($i=1,2$), it follows from \cite[III.6.3]{SM} that
the projective tensor prenorm
$\|\cdot\|_{\up_1}\Ptens\|\cdot\|_{\up_2}$ on
$\cA(M_1)\otimes\cA(M_2)$ is the Minkowski functional of
$$
S\!:=\acv\{\mu_1\otimes\mu_2\in \cA(M_1)\otimes
\cA(M_2):\,\mu_1\in V_{\up_1},\, \mu_2\in V_{\up_2})\}\,.
$$
 Since
$\|\cdot\|_{\up_1}$ and $\|\cdot\|_{\up_2}$ are continuous on
$\cA(M_1)$ and $\cA(M_2)$, resp.,
$\|\cdot\|_{\up_1}\Ptens\|\cdot\|_{\up_2}$ is extended to a
continuous prenorm on $\cA(M_1)\Ptens\cA(M_2)$, which coincides
with the Minkowski functional of $\overline{\rho(S)}$ via the
topological isomorphism $\cA(M_1)\Ptens\cA(M_2)\to\cA(M_1\times
M_2)$. It is not hard to see that $\overline{\rho(S)}$ equals to
$V_\up$, the closure of
$$
\acv\{\up(x_1,x_2)^{-1}\de_{(x_1,x_2)}:\,(x_1,x_2)\in M_1\times
M_2\}\,
$$
in $\cA(M_1\times M_2)$. Thus
$\|\cdot\|_{\up_1}\Ptens\|\cdot\|_{\up_2}$ and $\|\cdot\|_{\up}$
are identical.

Further, consider the projective system
$$
(\cA_{\up_1^n}(M_1)\Ptens \cA_{\up_2^m}(M_2)\!:(n,m)\in\N^2)
$$
(with naturally defined connecting maps) in the category of
Fr\'echet spaces. Since the diagonal is cofinal in $\N^2$, we
have
$$
\varprojlim_{(n,m)\in
\N^2}(\cA_{\up_1^n}(M_1)\Ptens\cA_{\up_1^m}(M_2))\cong
\varprojlim_{n\in \N}(\cA_{\up_1^n}(M_1)\Ptens
\cA_{\up_2^n}(M_2))\,.
$$
Writing $\varprojlim_{(n,m)}$ as an iterated projective limit and
using the fact that projective tensor products of Fr\'echet
spaces commute with projective limits, we get
$$
\Bigl(\varprojlim_{n\in
\N}\cA_{\up_1^n}(M_1)\Bigr)\Ptens\Bigl(\varprojlim_{m\in
\N}\cA_{\up_1^m}(M_2)\Bigr)\cong \varprojlim_{(n,m)\in
\N^2}(\cA_{\up_1^n}(M_1)\Ptens\cA_{\up_1^m}(M_2))\,.
$$
Thus
\begin{multline*}
\cA_{\up_1^\infty}(M_1)\Ptens \cA_{\up_2^\infty}(M_2)\cong
(\varprojlim_{n\in \N}\cA_{\up_1^n}(M_1))\Ptens(\varprojlim_{m\in \N}\cA_{\up_1^m}(M_2))\cong\\
\varprojlim_{n\in \N}(\cA_{\up_1^n}(M_1)\Ptens
\cA_{\up_2^n}(M_2))\cong\varprojlim_{n\in \N}
\cA_{\up^n}(M_1\times M_2)\cong \cA_{\up^\infty}(M_1\times M_2)\,.
\end{multline*}

(B) Now suppose that each $\cA_{\up_i^\infty}(M_i)$ is nuclear for
$i=1,2$. Then the space $\cA_{\up_1^\infty}(M_1)\Ptens
\cA_{\up_2^\infty}(M_2)$ is also nuclear. Whence these spaces are
reflexive and we have from \cite[Lem.~2.11]{ArAMN} that
$\cA_{\up_i^\infty}(M_i)'\cong \cO_{\up_i^\infty}(M_i)$ and
$\cA_{\up^\infty}(M_1\times M_2)'\cong \cO_{\up^\infty}(M_1\times
M_2)$.

Recall that for any nuclear Fr\'echet spaces $E$ and $F$, the
natural linear map $ E'\otimes F'\to (E\Ptens F)'$ induces the
topological isomorphism $E'\Ptens F' \cong (E\Ptens F)'$. Thus
\begin{multline*}
\cO_{\up^\infty}(M_1\times M_2)\cong\cA_{\up^\infty}(M_1\times
M_2)'\cong\\ (\cA_{\up_1^\infty}(M_1)\Ptens
\cA_{\up_2^\infty}(M_2))'\cong \cA_{\up_1^\infty}(M_1)'\Ptens
\cA_{\up_2^\infty}(M_2)'\cong\\ \cO_{\up_1^\infty}(M_1)\Ptens
\cO_{\up_2^\infty}(M_2)\,.
\end{multline*}
\end{proof}

Recall that each submultiplicative weight has the form
$\om(g)=e^{\ell(g)}$, where $\ell$ is a length function. This
correspondence allows to apply results of Section~\ref{ABWLF}. In
decomposition \eqref{thredec}, three types of length function
appear:
\begin{itemize}
\item
$\log(1+\ell)$, where $\ell$ is a word length function on a simply
connected nilpotent Lie group;
\item
a word length function on a simply connected nilpotent Lie group
itself;
\item
a word length function on a  connected linearly complex reductive
Lie group.
\end{itemize}

We look on these cases separately.

If $G$ is an affine algebraic complex group, then we consider the
algebra $\cR(G)$ of regular (in the sense of algebraic geometry)
functions as a $\Ptens$-algebra w.r.t. the strongest locally
convex topology. Note that  a simply connected nilpotent complex
Lie group $G$  is affine algebraic and $\cR(G)$ is just the algebra of
polynomials.

 \begin{lm}\label{RE}
Let $G$ be a simply connected nilpotent complex Lie group and
$\om(g)\!:=1+\ell(g)$, where $\ell$ is a word length function on
$G$. Then $\cO_{\om^\infty}(G)= \cR(G)$ as locally convex
algebras.
 \end{lm}
 \begin{proof}
If $\|\cdot\|$ is a norm on the Lie algebra $\fg$ of $G$, then,
by Lemma~\ref{loglogeq}, we have $\log(1+\ell(\exp\eta))\simeq
\log (1+\|\eta\|)$ on $\fg$. So $f \in\cO_{\om^\infty}(G)$ iff it
is bounded by a polynomial in norm.  Hence  $\cO_{\om^\infty}(G)=
\cR(G)$. Furthermore, the topology on $\cR(G)$ coincides with the
inductive topology of $\cO_{\om^\infty}(G)$.
\end{proof}

The case of a word length function  on a  simply connected
nilpotent complex Lie group is considered in \cite{ArAMN}. The
following result is [ibid. Th.~3.2].
\begin{thm}\label{exptypdesc}
Let $G$ be a simply connected nilpotent complex Lie group with
Lie algebra $\fg$, and let  $(t_1,\ldots, t_ m)$ be the canonical
coordinates of the first kind associated with an
$\mathscr{F}$-basis in $\fg$, where $\mathscr{F}$ is the lower
central series. Then
\begin{multline*}
\cO_{exp}(G)= \bigl\{f\in \cO(G):\, \\ \exists C>0,\, \exists
r\in \R_+ \, \text{s.t.}\,|f(t_1,\ldots, t_ m)|\le C
e^{r\max_i|t_i|^{1/w_i}}\,\forall t_1,\ldots, t_ m \bigr\}
\end{multline*}
and  we have
$$
\cO_{exp}(G)\cong \varinjlim_{r\in \R_+} \cO_{\eta^r}(G)
$$
as locally convex spaces, where $\eta(t_1,\ldots, t_
m)\!:=e^{\max_i|t_i|^{1/w_i}}$ and the Banach space
$\cO_{\eta^r}(G)$ is defined  as in~\eqref{cOom}.
\end{thm}

To consider the linearly complex reductive case we need the
following result, which is well known.
\begin{thm} \label{repolho}
Let  $L$  be  connected linearly complex reductive.  Then  any
holomorphic   homomorphism  of  $L$  into  a  complex algebraic  group  $H$
is  polynomial.
\end{thm}
The proof is similar to \cite[Th.~3.3.4]{ViOn}. The only step which is
different is that although we cannot claim that $[\fg,\fg]=\fg$ but
nevertheless  any reductive subalgebra in a complex Lie algebra is
algebraically closed.

Note that any connected linearly complex reductive group $L$ is an
affine algebraic group. This can be obtain, e.g.,  by application
of \cite[Ths.~2.23, 5.10]{Le02} from the fact that the algebra of
real analytic representative functions on $K$, where $L$ is a
universal complexification of  $K$, is finitely generated
\cite[Ch.~VI, \S~VII]{Ch46}.

\begin{thm}\label{redreg}
Suppose that $L$ is connected linearly complex reductive. Then
$\cO_{exp}(L)=\cR(L)$ as a locally convex algebra.
\end{thm}
\begin{proof}
Since $L$ is an affine algebraic, we have
$\cR(G)\subset\cO_{exp}(G)$   \cite[(5.31)]{Ak08}.

Consider a compact subgroup $K$ s.t. $L$ is the universal
complexification of~$K$. Note that the map $\mathscr{E}(K)'\to
\cA(L)$, which is dual to $\cO(L)\to \mathscr{E}(K)$, has dense
range \cite[Pr.~4]{Li70}.  Then it follows from Lemma~\ref{cLgrfg}
that the closure of the range of any homomorphism of $\cA(L)$ to a
Banach algebra  is classically semisimple.   In particular, if
$\om$ is a submultiplicative weight on $L$, then $\cA_\om(L)$  is
finite-dimensional.  By \cite[Lem.~2.10]{ArAMN}, we get
$\cO_\om(L)\cong \cA_\om(L)'$; hence $\cO_\om(L)$ is
finite-dimensional. Thus $\cO_{exp}(L)$ is an inductive limit of
finite-dimensional spaces, hence the topology on $\cO_{exp}(L)$ is
the strongest locally convex topology.

By Proposition~\ref{exthomo}, any $f\in\cO_{exp}(L)$ is a
coefficient of a holomorphic homomorphism to the invertibles of a
Banach algebra. As pointed out above, we can assume that this
Banach algebra is classically semisimple; so $f$ is a coefficient
of some holomorphic finite-dimensional representation. Moreover,
Theorem~\ref{repolho} implies that this representation is
polynomial; therefore $f\in \cR(L)$.  So $\cO_{exp}(L)\subset
\cR(L)$; thus $\cO_{exp}(L)=\cR(L)$ and the topologies coincide.
\end{proof}

Now we prove our main result.

\begin{thm}\label{fexpdec}
Let $G$ be a connected linear complex Lie group and $E$ is the
exponential radical of $G$. Fix a decomposition $G\cong B\rtimes
L$, where $B$  is simply connected nilpotent and $L$ is linearly
complex reductive. Then  the map $\tau$ defined in
Theorem~\ref{qitiG} induces an isomorphism of $\Ptens$-algebras
 $$
\cR(E) \Ptens \cO_{exp}(B/E)\Ptens \cR(L)\to \cO_{exp}(G)\,,
 $$
where $\cO_{exp}(B/E)$ is described in Theorem~\ref{exptypdesc}.
\end{thm}
\begin{proof}
Since $B/E$ and $E$ are nilpotent and simply connected,  the
exponential maps on $B/E$ and $E$ are  biholomorphic
equivalences. So we can consider $\tau$ as a map from $E\times
B/E\times L$ to $G$ (up to the identification of $\fv$ with
$\fb/\fe$).

Let $\ell$, $\ell_0$, $\ell_1$, and $\ell_2$ denote word length
functions on $G$, $E$, $B/E$, and $L$, resp.
Define submultiplicative weights
$$
\om(g)\!:=e^{\ell(g)},\quad \om_0(e)\!:=1+\ell_0(e),\quad
\om_1(h)\!:=e^{\ell_1(h)},\quad\om_2(l)\!:=e^{\ell_2(l)}
$$
on $G$, $E$, $B/E$, and $L$, resp. Note that $\cO_{exp}(G)$ is
topologically isomorphic to $\cO_{\om^\infty}(G)$ (see
\cite[Th.~4.3]{Ak08} or \cite[Pr.~2.8]{ArAMN}) . From the length
function equivalence given in Theorem~\ref{qitiG}, we have that
$\cO_{\om^\infty}(G)$ is topologically isomorphic to
$\cO_{\up^\infty}(E\times B/E\times L)$, where
$\up(e,h,l)\!:=\om_0(e)\om_1(h)\om_2(l)$.

Lemma~\ref{RE} implies that $\cO_{\om_0^\infty}(E)= \cR(E)$ as
locally convex spaces. Then $\cA_{\om_0^\infty}(E)$ is a space of
formal power series, which is nuclear. Furthermore,  being the
Arens-Michael envelopes of  $\cA(B/E)$ and $\cA(L)$, the
Fr\'{e}chet algebras $\cA_{\om_1^\infty}(B/E)$ and
$\cA_{\om_2^\infty}(L)$ are also nuclear \cite[Th.~5.10]{Ak08}.
Applying Proposition~\ref{Aom12}, we have
$$
\cO_{\up^\infty}(E\times B/E\times L)\cong
\cO_{\om_0^\infty}(E)\Ptens \cO_{\om_1^\infty}(B/E)\Ptens
\cO_{\om_2^\infty}(L)\,.
$$
Finally,  Theorem~\ref{redreg} implies that
$\cO_{\om_2^\infty}(L)= \cO_{exp}(L)=\cR(L)$.
\end{proof}

Thus the combination of Theorems~\ref{redliexpf}, \ref{exptypdesc},
and~\ref{fexpdec} gives a complete description of the algebra
$\cO_{exp}(G)$ for an arbitrary connected complex Lie group~$G$.

Now we turn to an application of these  results to a question on
the Arens-Michael envelope of $\cO_{exp}(G)$, which initially
motivated this research.

\begin{co}\label{densrexp}
The map $\tau$ defined in Theorem~\ref{qitiG} induce an
homomorphism
 $$
\te\!:\cR(E\times B/E \times L)\to \cO_{exp}(G)
 $$
 that has dense range.
\end{co}
\begin{proof}
Only density is left to prove. Theorem~\ref{exptypdesc} implies
that $\cR(B/E)$ (the polynomials) is dense in $\cO_{exp}(B/E)$.
Therefore the image of $\cR(E)\otimes \cR(B/E) \otimes \cR(L)$
under our homomorphism is dense in $\cO_{exp}(G)$.
\end{proof}

For a connected complex Lie group  $G$, we consider the
natural embeddings  $j\!:\cO_{exp}(G)\to \cO(G)$ and
$j_0\!:\cO_{exp}(G/\LinC(G))\to \cO(G/\LinC(G))$. Also, we remind
the reader that $\wt\si\!:\cO_{exp}(G/\LinC(G))\to \cO_{exp}(G) $
from Theorem~\ref{redliexpf} is a topological isomorphism.

\begin{thm}\label{expclin}
Let $G$ be a connected  complex Lie group. Then

{\em (A)} $j_0\wt\si^{-1}\!:\cO_{exp}(G)\to \cO(G/\LinC(G))$ is an
Arens-Michael envelope.

{\em (B)} $G$ is linear iff it is a Stein group and
$j\!:\cO_{exp}(G)\to \cO(G)$ is an Arens-Michael envelope.
\end{thm}
\begin{proof}
First, we show the condition from part~(B) is necessary. Suppose
that $G$ is linear. Then it is clearly a Stein group. Let $B$,
$L$, and $E$ be as above and note that $E\times B/E\times L$ is an
affine algebraic variety; so the obvious embedding
$$
\io\!:\cR(E\times B/E\times L)\to \cO(E\times B/E\times L)
$$
is an Arens-Michael envelope \cite[Ex.~3.6]{Pir_qfree}.
Identifying   $\cO(E\times B/E\times L)$  with $\cO(G)$, we obtain
that $\io=j\te$, where $\te$ is defined in
Corollary~\ref{densrexp}. The homomorphism~$\te$, having dense
range, is an epimorphism. It follows from \cite[Lem.~2.3]{ArAMN}
that the factorization  $\io=j\te$ of the Arens-Michael envelope
homomorphism $\io$ on the epimorphism $\te$ implies that $j$ is
also an Arens-Michael envelope homomorphism.

Next, to prove part~(A) note that $G/\LinC(G)$ is linear. The
above argument shows that $j_0$ is an Arens-Michael envelope, so
is $j_0\wt\si^{-1}$.

Finally, we demonstrate the sufficiency from Part~(B). Suppose
that $G$ is a Stein group and $j$ is an Arens-Michael envelope.
Since so is $j_0\wt\si^{-1}$, the universal property of  the
Arens-Michael enveloping functor  implies that $\cO(G/\LinC(G))\to
\cO(G) $ is a topological isomorphism of Stein algebras. By
Forster's Duality Theorem \cite{Fo67}, the quotient map $G\to
G/\LinC(G)$ is a biholomorphic equivalence, therefore $\LinC(G)$
is trivial.
\end{proof}

We finish with examples.

\begin{exm}
Let $\fg$ be the $2$-dimensional solvable complex Lie algebra
with basis $\{e_1,e_2 \}$ and  commutation relation
$[e_1,e_2]=e_2$. Then $\fe=[\fg,\fg]=\CC e_2$. Consider  the
simply connected complex Lie group $G$ with $\fg$ as the Lie
algebra (cf.~\eqref{2dimmul}).  Let
 $(s,t)$ be the canonical coordinates of second type on $G$, i.e.,
$$
g=\exp(se_1)\exp(te_2) \qquad(g\in G)\,.
$$
In this coordinates, any $f\in\cO_{exp}(G)$ has the form
$$
f(s,t)=\sum_{n} f_{n}(s) t^{n}\,,
$$
where each $f_{n}$ is an entire function of exponential type on $\CC$, i.e.,
$$
|f_{n}(s)|\le C e^{r|s|}\qquad (s\in\CC)
$$
for some $C>0$ and $r\in \R_+$. Note that this decomposition
can also be obtain from \cite[Prop.~5.2]{Pir_qfree}.

Another group with the same Lie algebra is the '$az+b$'-group $G_1$,
consisting of matrices of the form
$$
 \begin{pmatrix}a& b \\
 0&1
\end{pmatrix}, \qquad (a\in\CC^\times,\,b\in\CC)\,.
$$
Unlike $G$, the group $G_1$ is algebraic and
$\cO_{\exp}(G_1)=\cR(G_1)$, i.e., a function of exponential type
has a decomposition $f(a,b)=\sum_{n} f_{n}(a) b^{n}$ where
all $f_n$ are Laurent polynomials in $a$.
\end{exm}

\begin{exm} \label{6dim}
Let $\fg$ be the $6$-dimensional complex Lie algebra with basis
$\{e_1,e_2 ,e_3, f_1,f_2,f_3\}$ and commutation relations
\begin{gather*}
[e_1,e_2]=e_3 ,\\
[e_2,f_1]=f_1,\; [e_2,f_2]=f_2,\;[e_2,f_3]=2f_3, \\
[e_3,f_1]=f_1,\; [e_3,f_2]=-f_2,\;[e_3,f_3]=0, \\
[f_1,f_2]=f_3
\end{gather*}
the  undefined brackets being zero. To see that $\fg$ is a Lie
algebra it is sufficient to note that $\fg$  is an iterated
semidirect sum $\fh_1\ltimes (\fh_2\ltimes \fh_3)$, where
$\fh_1=\spn\{e_1\}$, $\fh_2=\spn\{e_2,e_3\}$, and
$\fh_3=\spn\{f_1,f_2,f_3\}$.

Then $e_3, f_1,f_2,f_3$ is a basis for $\fg_2\!:=[\fg,\fg]$ but
$f_1,f_2,f_3$ is a basis for $\fg_\infty=\fg_3\!:=[\fg,\fg_2]$.
Since $\fg$ is solvable, we have $\fe= \fg_\infty$. So $\fe$ and
$\fg/\fe$ are both isomorphic to the $3$-dimensional complex
Heisenberg algebra.

Let $G$ be the simply connected complex Lie group having $\fg$ as
the Lie  algebra. Consider the  coordinates
$(s_1,s_2,s_3,t_1,t_2,t_3)$ defined by
$$
g=\exp(s_1e_1+s_2e_2+s_3e_3)\exp(t_1f_1+t_2f_2+t_3f_3) \qquad(g\in G)
$$
and identify $G$ with $\CC^6$.  Thus any $f\in\cO_{exp}(G)$ has the form
$$
f(s_1,s_2,s_3,t_1,t_2,t_3)=\sum_{n_1,n_2,n_3}
f_{n_1,n_2,n_3}(s_1,s_2,s_3) t_1^{n_1} t_2^{n_2} t_3^{n_3}\,,
$$
where each $f_{n_1,n_2,n_3}$ is an entire function such that
$$
|f_{n_1,n_2,n_3}(s_1,s_2,s_3)|\le C
e^{r\max\{|s_1|,|s_2|,|s_3|^{1/2}\}}\qquad (s_1,s_2,s_3\in\CC)
$$
for some $C>0$ and $r\in \R_+$ (cf. \cite[Exm.~3.4]{ArAMN}).
\end{exm}

\begin{exm}
Fix $n\in\N$ and consider the standard action of $\SL_n(\CC)$ on
$\CC^n$. Set $G\!:=\CC^n\rtimes \SL_n(\CC)$. The Lie algebra of
$G$ is the semidirect sum $\fg=\CC^n\rtimes \mathfrak{sl}_n(\CC)$.
Then the radical $\fr\cong  \CC^n$ and  $\mathfrak{sl}_n(\CC)$ is
a Levi complement. It is easy to see that $\fr_\infty=0$ but $\fe
\cong \CC^n$; so $\fr_\infty\ne \fe$. Thus
$\cO_{exp}(G)=\cR(\CC^n\times \SL_n(\CC))$, i.e., every
holomorphic function of exponential type is a polynomial in
coordinates on $\CC^n$ and matrix elements of $\SL_n(\CC)$.
\end{exm}

The reader can also find another example in
\cite[Exm.~3.5]{ArAMN}.

\end{document}